\documentclass[letterpaper,11pt]{amsart}

\usepackage[margin=1.2in]{geometry}
\usepackage{amsmath,amsthm,amssymb}
\usepackage{xspace,xcolor}
\usepackage[breaklinks,colorlinks,citecolor=teal,linkcolor=teal,urlcolor=teal,pagebackref,hyperindex]{hyperref}
\usepackage[alphabetic]{amsrefs}
\usepackage[all]{xy}
\usepackage{color}

\theoremstyle{plain}
\newtheorem{thm}{Theorem}[section]

\newtheorem{lem}[thm]{Lemma}
\newtheorem{prop}[thm]{Proposition}
\newtheorem{cor}[thm]{Corollary}

\theoremstyle{definition}

\newtheorem{eg}[thm]{Example}

\theoremstyle{remark}
\newtheorem{rmk}[thm]{Remark}

\def\Q{{\mathbf Q}}

\def\C{{\mathbf C}}

\def\cD{\mathcal{D}}

\def\cI{\mathcal{I}}

\def\cM{\mathcal{M}}

\def\cO{\mathcal{O}}

\def\.{\cdot}
\def\^{\widehat}

\def\({\left(}
\def\){\right)}

\renewcommand{\and}{ \ \ \text{ and } \ \ }

\DeclareMathOperator{\lct} {lct}

\begin{document}

\author{Bradley Dirks}
\author{Mircea Musta\c{t}\u{a}}

\address{Department of Mathematics, University of Michigan, 530 Church Street, Ann Arbor, MI 48109, USA}

\email{bdirks@umich.edu}
\email{mmustata@umich.edu}

\thanks{The authors were partially supported by NSF grant DMS-1701622.}


\begin{abstract}
By building on a method introduced by Kashiwara \cite{Kashiwara} and refined by Lichtin \cite{Lichtin},
we give upper bounds for the roots of certain $b$-functions associated to a regular function $f$
in terms of a log resolution of singularities. As applications, we recover with more elementary
methods a result of Budur and Saito \cite{BS} describing the multiplier ideals of $f$ in terms of
the $V$-filtration of $f$ and a result of the second named author with Popa \cite{MP}
giving a lower bound for the minimal exponent of $f$ in terms of a log resolution.
\end{abstract}

\title{Upper bounds for roots of $B$-functions, following Kashiwara and Lichtin}

\maketitle

\section{Introduction} 

 Given a nonzero regular function $f\in \mathcal O_X(X)$ on a smooth complex variety $X$, the \emph{Bernstein-Sato polynomial}
(or \emph{$b$-function}) of $f$ is the monic polynomial 
 $b_f(s)$ of minimal degree that satisfies
 \begin{equation}\label{eq_def_b_function}
 b_f(s) f^s \in {\mathcal D}_X[s]\cdot f^{s+1},
 \end{equation} where $s$ is an indeterminate
 and ${\mathcal D}_X$ is the sheaf of differential operators on $X$. Here $f^s$ can be treated as a formal symbol on which differential
 operators on $X$ act in the expected way. For example, if $f$ defines a nonempty smooth hypersurface, then $b_f(s)=s+1$.
 The $b$-function is an important invariant of singularities introduced independently by Bernstein \cite{Bernstein} and Sato. 
 
 Kashiwara \cite{Kashiwara} showed that the roots of $b_f(s)$ are negative rational numbers.
By refining Kashiwara's approach, Lichtin proved in \cite[Theorem~5]{Lichtin} the following estimate for the roots of $b_f(s)$ in terms of a 
strong log resolution of the pair $(X,D)$, where $D$ is the divisor defined by $f$. 
By this we mean a projective morphism
$\pi\colon Y\to X$ which is an isomorphism over the complement of the support of $D$, such that $Y$ is smooth
and $\pi^*D=\sum_{i=1}^ra_iE_i$ is a simple normal crossing divisor. Note that we can also write 
$K_{Y/X}=\sum_{i=1}^rk_iE_i$, where 
$K_{Y/X}$ is the relative canonical divisor, locally defined by the Jacobian of the morphism $\pi$.

\begin{thm}[Lichtin]\label{thm_Lichtin}
With the above notation, every root of $b_f(s)$ is of the form $-\frac{k_i + 1+ \ell}{a_i}$ for some 
$i$ and some nonnegative integer $\ell$.
\end{thm}

A consequence of the above theorem is that every root of $b_f(s)$ is bounded above by $-\min_i\frac{k_i+1}{a_i}$. 
We note that the invariant $\min_i\frac{k_i+1}{a_i}$ is independent of the choice of log resolution; it is the \emph{log canonical threshold}
${\rm lct}(f)$ of $f$. Using an argument based on integration by parts that goes back to \cite{Bernstein}, Koll\'{a}r showed in \cite[Theorem~10.6]{Kollar} that 
$-\lct(f)$ is a root of
$b_f(s)$; therefore the largest root is precisely $-\lct(f)$.
For basic facts about log canonical thresholds, we refer to \cite[Chapter 9]{Lazarsfeld}.

In this note we follow the same approach to give similar estimates for other $b$-functions related to $f$.
These are associated to certain elements of 
$$\widetilde{B}_f:=\cO_X[1/f,s]f^s.$$
Note that this is a ${\mathcal D}_X\langle t,\partial_t\rangle$-module, with $t$ acting as the automorphism that maps $P(s)f^s$ to $P(s+1)f^{s+1}=\big(P(s+1)f\big)f^s$,
and $\partial_t$ acting as $-st^{-1}$.
The module $\widetilde{B}_f$ plays an important role in Malgrange's description in \cite{Malgrange} of the nearby cycle sheaf of $f$ on the level of 
${\mathcal D}$-modules. Using the existence of the Bernstein-Sato polynomial and the rationality of its roots, Malgrange constructed the
$V$-filtration on $\widetilde{B}_f$, whose successive quotients are related to the nearby cycles sheaf (see the next section for a few more details about the $V$-filtration). 

The existence of the $V$-filtration easily implies the existence of a $b$-function for every element $u\in \widetilde{B}_f$. This is the monic polynomial of
smallest degree such that 
\begin{equation}\label{eq_def_gen_b_function}
b_u(s)u\in {\mathcal D}_X\langle t,\partial_tt\rangle \cdot tu.
\end{equation}
In turn, it was shown by Sabbah \cite{Sabbah} that the $V$-filtration can be described in terms of the roots of $b$-functions
(see Theorem~\ref{thm_Sabbah} below for the precise statement). We note that with this notation, the Bernstein-Sato polynomial $b_f(s)$
is equal to $b_{u}(s)$, where $u=f^s$.

We are concerned with the roots of the $b$-function for elements of $\widetilde{B}_f$ of the form $g\partial_t^mf^s$, where $g\in\cO_X(X)$
is another nonzero regular function on $X$ and $m$ is a nonnegative integer. Our main result is the 
following extension of Lichtin's theorem. 
For every log resolution of $f$ as above, we put $b_i={\rm ord}_{E_i}(g)$, where ${\rm ord}_{E_i}$ is the valuation associated
to the divisor $E_i$ on the resolution.

\begin{thm}\label{rootbound}
With the above notation, if $u=g\partial_t^mf^s$, then the following hold:
\begin{enumerate}
\item[i)] Every root of $b_u$ is $\leq -\min\left\{1, -m+\min_i\frac{k_i+1+b_i}{a_i}\right\}$.
\item[ii)] If $m=0$, then every root of $b_u$ is 
$\leq -\min\left\{\frac{k_i+1+b_i}{a_i}\mid 1\leq i\leq r\right\}$.
\item[iii)] If $g=1$, then every root of $b_u$ is either a negative integer or of the form
$m-\frac{k_i+1+\ell}{a_i}$
for some $i$ and some nonnegative integer $\ell$. Furthermore, if we assume in addition that the divisor $D$ defined by $f$ is reduced and the strict transform
$\widetilde{D}$ of $D$ on $Y$ is smooth, we may only consider those $i$ with $E_i$ exceptional.
\end{enumerate}
\end{thm}

We note that while the result in i) covers the most general situation, the sharper upper bounds in ii) and iii) are the ones
that we will need for applications. The first application concerns $b$-functions of the form $b_{gf^s}$. 
Given $f$ and $g$ as above, let us put
$${\rm lct}_g(f)=\min\{\lambda>0\mid g\not\in {\mathcal I}(f^{\lambda})\},$$
where ${\mathcal I}(f^{\lambda})$ denotes the multiplier ideal of $f$, with exponent $\lambda$
(for basic facts about multiplier ideals, see \cite[Chapter 9]{Lazarsfeld}). Note that for $g=1$, we recover
${\rm lct}(f)$. It is an immediate consequence of the definition of multiplier ideals that we have
$${\rm lct}_g(f)=\min_i\frac{k_i+1+b_i}{a_i}.$$
We thus see that Theorem~\ref{rootbound}ii) implies that every root of $b_{gf^s}$ is bounded above by $-{\rm lct}_g(f)$.

On the other hand, recall that Koll\'{a}r's integrability argument was refined in \cite{ELSV} to show that every jumping number of $f$
that lies in $(0,1]$ is a root of $b_f$. A small modification of the argument in \emph{loc. cit.} allows us to show that 
$-{\rm lct}_g(f)$ is a root 
of $b_{gf^s}$. We thus have

\begin{cor}\label{equality}
Given $f$ and $g$ as above, the largest root of $b_{gf^s}$ is $-\lct_g(f)$.
\end{cor}

Using Sabbah's description of the $V$-filtration $V^{\bullet}\widetilde{B}_f$ on $\widetilde{B}_f$ in terms of roots of $b$-functions, this translates as 
the following result due to Budur and Saito, see \cite[Theorem~0.1]{BS}.

\begin{cor}\label{BSthm}
For every nonzero $f\in\cO_X(X)$ and every positive $\alpha\in\Q$, we have
$$\{g\in\cO_X\mid gf^s\in V^{\alpha}\widetilde{B}_f\}={\mathcal I}(f^{\alpha-\epsilon})\quad\text{for}\quad 0<\epsilon\ll 1.$$
\end{cor}

Our second application of Theorem~\ref{rootbound} is towards a lower bound for the minimal exponent $\widetilde{\alpha}(f)$ of $f$. 
Recall that if $f$ is not invertible, then it follows easily from (\ref{eq_def_b_function}) that $b_f(-1)=0$. The negative of the largest root of $b_f(s)/(s+1)$
is the \emph{minimal exponent} $\widetilde{\alpha}(f)$ (the usual convention is that if $b_f(s)=s+1$, which is the case
precisely when $f$ defines a smooth hypersurface, then $\widetilde{\alpha}(f)=\infty$). Note that by the result of Lichtin and Koll\'{a}r discussed above, we have
${\rm lct}(f)=\min\{\widetilde{\alpha}(f),1\}$. 

Using a result due to Saito \cite{Saito-MLCT} which describes $\widetilde{\alpha}(f)$ in terms of the roots of $b_{\partial_t^mf^s}$, we obtain the following 
lower bound for $\widetilde{\alpha}(f)$. We assume that $f$ defines a \emph{reduced} nonzero divisor $D$ and we assume that the strong log resolution 
$\pi\colon Y\to X$ has the property that the strict transform $\widetilde{D}$ is smooth (note that we can obtain such a resolution from an arbitrary one by
performing some extra blow-ups).

\begin{cor}\label{bound_min_exp}
With the above notation, we have
$$\widetilde{\alpha}(f)\geq\min_{i; E_i \text{exceptional}}\frac{k_i+1}{a_i},$$
where the minimum is over the exceptional divisors $E_i$. 
\end{cor}

This result was proved by the second named author with Popa 
using the theory of Hodge ideals, see
\cite[Corollary~D]{MP}. In fact, the bound follows easily from Lichtin's Theorem~\ref{thm_Lichtin} when $\widetilde{\alpha}(f)$ is 
not an integer, but it does not seem clear how to deduce this in general from that result. We note that while the original proofs of the results in Corollaries~\ref{BSthm}
and \ref{bound_min_exp}
made use of the deep results in Saito's theory of mixed Hodge modules \cite{Saito-MHM}, the arguments in this note only rely on basic results in the theory 
of $b$-functions. 

In the next section we review briefly material related to the $\cD$-module $\widetilde{B}_f$, the $V$-filtration, and $b$-functions. We also include here some easy lemmas
on $b$-functions.
The proof of the main result in Theorem~\ref{rootbound} is given in Section~3. The application to the description of the multiplier ideals via the $V$-filtration 
is discussed in Section~4, while the bound for the minimal exponent is deduced in Section~5.

\subsection{Acknowledgements} The second author is grateful to Mihnea Popa for many discussions related to $b$-functions and $V$-filtrations.

\section{The $\cD$-module $\widetilde{B}_f$}

Let $X$ be a smooth, $n$-dimensional complex algebraic variety. We denote by ${\mathcal D}_X$ the sheaf of differential operators on $X$ and by
$\cD_X\langle t,\partial_t\rangle$ the push-forward to $X$ of the sheaf of differential operators on $X\times {\mathbf A}^1$ (hence $\partial_t$ and $t$
satisfy the commutation relation $[\partial_t,t]=1$ and they commute with the sections of $\cD_X$). We will also consider the subsheaf
$\cD_X\langle \partial_tt, t\rangle$ of $\cD_X\langle t,\partial_t\rangle$ generated over $\cD_X$ by $\partial_tt$ and $t$. 
For basic facts about $\cD$-modules, we refer to \cite{HTT}.

Recall from the introduction that $\widetilde{B}_f=\cO_X[1/f,s]f^s$. This is a free module of rank $1$ over the sheaf $\cO_X[1/f,s]$ of polynomials in $s$ with coefficients
in $\cO_X[1/f]$. Note that $\widetilde{B}_f$ has a natural structure of left $\cD_X$-module, with a derivation $D\in {\mathcal Der}_{\bf C}(\cO_X)$ acting on $f^s$ in the 
expected way:
$$D\cdot f^s=\frac{sD(f)}{f}f^s.$$
This extends to a left action of $\cD_X\langle \partial_tt,t\rangle$ on $B_f$, with $t$ acting by the automorphism
$$P(s)f^s\to P(s+1)f^{s+1}:=\big(P(s+1)f\big)f^s$$
and $-\partial_tt$ acting by multiplication with $s$ (because of this, we also denote by $s$ the operator $-\partial_tt$). 
In order for this to be well-defined, we only need to check that the operators on $\widetilde{B}_f$ that we defined satisfy the commutation relation 
$st=t(s-1)$, which is an easy exercise. Finally, since $t$ acts by an automorphism, we can make $\widetilde{B}_f$ a left module over $\cD_X\langle t,\partial_t\rangle$
by letting $\partial_t$ act as $-st^{-1}$. 

For future reference, we note that for every polynomial $Q$ in one variable we have
\begin{equation}\label{eq1_lem2}
Q(s)t=tQ(s-1)\quad\text{and}\quad \partial_tQ(s)=Q(s-1)\partial_t.
\end{equation}
Indeed, it is enough to check the equalities for $Q(s)=s^j$, when these follow by an easy induction on $j$.

\begin{rmk}
It is well-known (and not hard to see) that $\widetilde{B}_f$ is a free $\cO_X[1/f]$-module, with a basis given by $\partial_t^mf^s$, for $m\geq 0$. 
On this basis, the action of $t$ is given by
\begin{equation}\label{action_of_t}
t\cdot \partial_t^mf^s=f\partial_t^mf^s-m\partial_t^{m-1}f^s
\end{equation}
and the action of a derivation $D\in {\mathcal Der}_{\C}(\cO_X)$ is given by 
$$D\cdot \partial_t^mf^s=-D(f)\partial_t^{m+1}f^s.$$
If we put $B_f:=\bigoplus_{m\geq 0}\cO_X\partial_t^mf^s$, then $B_f$ is a $\cD_X\langle t,\partial_t\rangle$-submodule of $\widetilde{B}_f$.
\end{rmk}

Malgrange constructed in \cite{Malgrange} the \emph{$V$-filtration} $(V^{\alpha}\widetilde{B}_f)_{\alpha}$ on $\widetilde{B}_f$, a certain decreasing filtration parametrized by rational numbers\footnote{Actually, the filtration constructed in \cite{Malgrange} was parametrized by integers. The filtration indexed by ${\mathbf Q}$ was obtained (for more general $\cD$-modules) in \cite{Saito-GM}.}, and uniquely characterized by a certain list of axioms. 
The construction made use of the existence of the Bernstein-Sato polynomial of $f$ and of the rationality of its roots. Using the existence of the $V$-filtration,
one can show that $b$-functions can be associated to arbitrary elements of $\widetilde{B}_f$; moreover, the $V$-filtration can be characterized in terms of the
roots of the general $b$-functions (see Theorem~\ref{thm_Sabbah} below).

Recall from the Introduction that given a section $u$ of $\widetilde{B}_f$, the \emph{$b$-function} of $u$ is the monic polynomial of minimal 
degree $b_u\in {\mathbf C}[s]$ such that 
$$b_u(s)u\in \cD_X\langle s,t\rangle \cdot tu.$$
As we have already mentioned, its existence can be easily deduced from the existence of the $V$-filtration on $\widetilde{B}_f$; moreover, this argument
also implies that all roots of $b_u$ are rational numbers. 

\begin{rmk}\label{rmk_equiv_b_fcn}
Note that if $u=gf^s$ for some $g\in\cO_X(X)$, then $t^ju=f^ju$ for every $j\geq 0$. This implies that $b_u(s)$ is the monic polynomial of 
smallest degree such that
$$b_u(s)u\in\cD_X[s]\cdot fu.$$
In particular, we see that $b_{f^s}$ is the Bernstein-Sato polynomial $b_f$. 
\end{rmk}

\begin{rmk}\label{local_char}
It is clear from the definition of $b$-functions that if $u$ is a section of $\widetilde{B}_f$ on an open subset $V$ of $X$ and if we have
an open cover $V=\bigcup_iV_i$, then $b_u(s)$ is the least common multiple of the polynomials $b_{u\vert_{V_i}}(s)$, for $i\in I$. 
\end{rmk}

Concerning the $V$-filtration, we will only need the following characterization, 
due to Sabbah \cite{Sabbah}.

\begin{thm}\label{thm_Sabbah}
For every $\alpha\in{\mathbf Q}$, we have
$$V^{\alpha}\widetilde{B}_f=\{u\in \widetilde{B}_f\mid \text{all roots of}\,\,b_u(s)\,\text{are}\,\leq -\alpha\}.$$
\end{thm}

\begin{rmk}\label{reduced_b_function}
Suppose that $g\in\cO_X(X)$ is such that $g/f$ is not a regular function. In this case, $b_{gf^s}(-1)=0$. 
Indeed, it follows from Remark~\ref{rmk_equiv_b_fcn} that we have 
$$b_{gf^s}(s)gf^s\in \cD_X[s]\cdot (gf)f^s.$$
By specializing to $s=-1$, we get $b_{gf^s}(-1)g/f\in \cD_X\cdot g\subseteq\cO_X$.
Since $g/f$ is not a regular function, it follows that $b_{gf^s}(-1)=0$. The \emph{reduced $b$-function} of $gf^s$
is 
$$\widetilde{b}_{gf^s}(s)=b_{gf^s}(s)/(s+1).$$
\end{rmk}

We now give a few easy results about $b$-functions that will be needed in the next section for the 
proof of our main result. We begin with the following lemma, which extends the well-known fact that the $b$-function of $f$ only depends on the ideal generated by $f$.

\begin{lem}\label{lem1}
Let $p,q\in\cO_X(X)$ be invertible functions and let $h\in\cO_X(X)$ be nonzero. If 
$g=pf$ and we consider 
$u=h\partial_t^mf^s\in \widetilde{B}_f$ and $v=qh\partial_t^mg^s\in \widetilde{B}_g$, then $b_u=b_v$.
\end{lem}

\begin{proof}
Let $\widetilde{B}_f^{\star}$ be equal to $\widetilde{B}_f$ as a sheaf of $\cO_X$-modules, but with a new $\cD_X\langle t,\partial_t\rangle$-action, 
denoted $\star$, given for every $\beta\in \widetilde{B}_f$ by
\begin{enumerate}
\item[i)] $D\star \beta=\big(D+sD(p)p^{-1}\big)\beta$ for every $D\in {\mathcal Der}_{\C}(\cO_X)$.
\item[ii)] $t\star\beta=(pt)\beta$.
\item[iii)] $\partial_t\star\beta=(p^{-1}\partial_t)\beta$.
\end{enumerate}
It is easy to check that this gives indeed a $\cD_X\langle t,\partial_t\rangle$-action.
Note that the new action of $s$ coincides with the old one.

It is straightforward to check that the map
$$\nu\colon \widetilde{B}_g\to \widetilde{B}_f^{\star},\,\, \nu\big(P(s)g^s\big)=P(s)f^s$$
is an isomorphism of $\cD_X\langle t,\partial_t\rangle$-modules that maps
$v$ to $qp^{-m}h\partial_t^mf^s$. It follows that $b_v(s)$ is the monic polynomial $b(s)$ of minimal degree that satisfies
$$b(s)qp^{-m}h\partial_t^mf^s\in \cD_X\langle s,t\rangle t\star qp^{-m}h\partial_t^{m}f^s.$$
Since for every $w\in \widetilde{B}_f$, we have
$$\cD_X\langle s,t\rangle t\star w=\cD_X\langle s,t\rangle tw$$
and for every invertible $\varphi\in\cO_X(X)$, we have
$$\cD_X\langle s,t\rangle\varphi w=\cD_X\langle s,t\rangle w=\varphi\cD_X\langle s,t\rangle w,$$
we deduce that $b(s)=b_u(s)$.
\end{proof}

\begin{lem}\label{lem2}
If $g\in\cO_X(X)$ is such that $\frac{g}{f}$ is not a regular function, then for every $m\geq 0$, we have
$$b_{g\partial_t^mf^s}(s)\vert (s+1)\widetilde{b}_{gf^s}(s-m).$$
\end{lem}

\begin{proof}
We follow the argument in the proof of \cite[Proposition~6.12]{MP}, which treats the case $g=1$.
Let $b(s)=b_{gf^s}(s)$ and $\widetilde{b}(s)=\widetilde{b}_{gf^s}(s)=b(s)/(s+1)$. Without any loss of generality, we may assume that $X$
is affine. By definition, there is $P\in\cD_X(X)\langle s,t\rangle$ such that
\begin{equation}\label{eq2_lem2}
b(s)gf^s=P\cdot t(gf^s).
\end{equation}
Since for every $j>0$ we have $t^j(gf^s)=(f^jg)f^s$, we may assume that $P\in\cD_X(X)[s]$. 

The assertion is clear for $m=0$, hence we may and will assume $m\geq 1$.
Since $s+1=-t\partial_t$, we deduce from (\ref{eq2_lem2}), using also (\ref{eq1_lem2}) that
$$-t\partial_t\widetilde{b}(s)gf^s=tP(s-1)gf^s.$$
Since the action of $t$ on $\widetilde{B}_f$ is injective, we conclude that
$$-\widetilde{b}(s-1)\partial_tgf^s=-\partial_t\widetilde{b}(s)gf^s=P(s-1)gf^s,$$
where the first equality follows from (\ref{eq1_lem2}). 
Furthermore, using again the relations (\ref{eq1_lem2}), we obtain
$$(s+1)\widetilde{b}(s-m)g\partial_t^mf^s=(s+1)\partial_t^{m-1}\widetilde{b}(s-1)\partial_tgf^s$$
$$=-(s+1)\partial_t^{m-1}P(s-1)gf^s=-(s+1)P(s-m)\partial_t^{m-1}gf^s=P(s-m)tg\partial_t^mf^s.$$
Since 
$$(s+1)\widetilde{b}(s-m)g\partial_t^mf^s\in\cD_X[s]tg\partial_t^mf^s,$$
the assertion in the lemma follows from the definition of $b_{g\partial_t^mf^s}(s)$.
\end{proof}

In the next section we prove Theorem~\ref{rootbound} by reducing it to a computation in the simple normal crossing case.
This computation is the content of the next lemma. We assume here that we have global algebraic coordinates $x_1,\ldots,x_n$ on $X$ (that is,
$dx_1,\ldots,dx_n$ give a trivialization of the cotangent sheaf $\Omega_X$).

\begin{lem}\label{lem3}
With the above notation, suppose that $f=\prod_{i=1}^nx_i^{a_i}$ and $g=\prod_{i=1}^nx_i^{b_i}$. If $u=g\partial_t^mf^s$,
for some $m\geq 0$, then the following hold:
\begin{enumerate}
\item[i)] $b_u(s)$ divides  $(s+1)\cdot \prod_{i=1}^n\prod_{j=1}^{a_i}\left(s-m+\frac{b_i+j}{a_i}\right)$  (with the convention that the second product is $1$ 
if $a_i=0$).
\item[ii)] If $m=0$, then $b_u(s)$ divides $\prod_{i=1}^n\prod_{j=1}^{a_i}\left(s+\frac{b_i+j}{a_i}\right)$.
 \item[iii)] For every $m$, if $a_1=1$ and $b_1=0$, then $b_u(s)$ divides $(s+1)\cdot \prod_{i=2}^n\prod_{j=1}^{a_i}\left(s-m+\frac{b_i+j}{a_i}\right)$.
\end{enumerate}
\end{lem}

\begin{proof}
It is convenient to write $gf^s$ as $\prod_ix_i^{a_is+b_i}$ and 
$tgf^s$ as $\prod_ix_i^{a_i(s+1)+b_i}$; we further write these using multi-index notation as $x^{as+b}$ and $x^{a(s+1)+b}$, respectively.
We begin by proving i) and ii). Let
$$h(s)=\prod_{i=1}^n\prod_{j=1}^{a_i}\left(a_i(s-m)+b_i+j\right).$$
Note that we have 
$$\partial_{x_1}^{a_1}\cdots\partial_{x_n}^{a_n}\partial_t^mx^{a(s+1)+b}=\partial_t^m\cdot\prod_{i=1}^n\prod_{j=1}^{a_i}
(a_is+b_i+j)x^{as+b}=h(s)\partial_t^m x^{as+b},$$
where the last equality follows from (\ref{eq1_lem2}).
If $m=0$, this implies that
$$h(s)u\in\cD_X[s]\cdot tu,$$
hence $b_u(s)$ divides $h(s)$, giving the assertion in ii). 
In general, we have 
$$(s+1)h(s)u=\partial_{x_1}^{a_1}\cdots\partial_{x_n}^{a_n}(s+1)\partial_t^mtgf^s$$
and using the fact that 
$s+1=-t\partial_t$ and (\ref{eq1_lem2}), we see that
$$(s+1)\partial_t^mt=-t\partial_t^{m+1}t=t\partial_t^ms=t(s-m)\partial_t^m=(s-m+1)t\partial_t^m.$$
We thus conclude that 
$$(s+1)h(s)u=(s-m+1)\partial_{x_1}^{a_1}\cdots\partial_{x_n}^{a_n}tu\in\cD_X[s]tu,$$
and the assertion in i) follows from the definition of $b_u$.

We now turn to the proof of iii). If $v=gf^s$, then it follows from our assumption on $a_1$ and $b_1$ that $\widetilde{b}_v$ is well-defined
(see Remark~\ref{reduced_b_function}). Moreover, it follows from Lemma~\ref{lem2} that 
$b_u(s)$ divides $(s+1)\widetilde{b}_v(s-m)$, hence it is enough to show that
$\widetilde{b}_v(s)$ divides $\prod_{i=2}^n\prod_{j=1}^{a_i}\left(s+\frac{b_i+j}{a_i}\right)$. By the assumption on $a_1$ and $b_1$,
this is equivalent with
the fact that
$b_v(s)$ divides the polynomial $\prod_{i=1}^n\prod_{j=1}^{a_i}\left(s+\frac{b_i+j}{a_i}\right)$, which we have already seen in ii).
\end{proof}

\section{Proof of the main theorem}

We begin with some general considerations on the setting of Theorem~\ref{rootbound}. Suppose that $X$ is a smooth, $n$-dimensional
complex algebraic variety and $f\in\cO_X(X)$ is nonzero. 
For every $m\geq 0$ and $g\in\cO_X(X)$, we consider
$$N_{f,m}(g):=\cD_X\langle s,t\rangle\cdot g\partial_t^mf^s\subseteq \widetilde{B}_f.$$
If $g=1$, then we simply write $N_{f,m}$ instead of $N_{f,m}(g)$. 

Note that by (\ref{eq1_lem2}), we have
\begin{equation}\label{eq1_lem2_v2}
P(s,t)t=tP(s-1,t)\quad\text{for every}\quad P\in\cD_X\langle s,t\rangle.
\end{equation}
This immediately implies that $tN_{f,m}(g)$ is a $\cD_X\langle s,t\rangle$-submodule of $N_{f,m}(g)$.

\begin{rmk}\label{rem_interpretation_b_function}
We recall the following useful interpretation of $b$-functions: if $u=g\partial_t^mf^s$, then $b_u(s)$
is the minimal polynomial of the action of $s$ on the quotient $N_{f,m}(g)/tN_{f,m}(g)$. Indeed, it is clear from
(\ref{eq1_lem2_v2}) that $b(s)u\in\cD_X\langle s,t\rangle tu$ if and only if $b(s)u\in tN_{f,m}(g)$. 
Moreover, if this is the case, then for every $P\in\cD_X\langle s,t\rangle$, we have $b(s)P(s,t)u\in tN_{f,m}(g)$. 
\end{rmk}

Recall that the Jacobian ideal $J_f$ of $f$ is the coherent ideal of $\cO_X$ defined as follows. If $x_1,\ldots,x_n$ are algebraic coordinates 
on an affine open subset of $X$, then $J_f$ is generated in $U$ by $\partial f/\partial x_1,\ldots,\partial f/\partial x_n$ (the definition is independent
of the choice of coordinates, but it depends on $f$, not just on the ideal generated by $f$). We always make the following extra assumption on $f$:
\begin{equation}\label{eq_assumption}
\text{The zero locus of $J_f$ is contained in the zero-locus of $f$.}
\end{equation}
If $f$ is not a constant, then it follows by Generic Smoothness that this assumption is satisfied after possibly replacing $X$ by an open
neighborhood of the zero-locus of $f$. Such a replacement is harmless when studying the singularities of the hypersurface defined by $f$.

We note that by \cite[Theorem~5.3]{Kashiwara}, the $\cD_X$-module $N_{f,0}$ is coherent and its characteristic variety
is the closure $W_f$ of
$$\{\big(x,sdf(x)\big)\in T^*X\mid f(x)\neq 0, s\in\C\}.$$
It is clear from the definition that $W_f$ is an irreducible subvariety of $T^*X$, of dimension $n+1$, which dominates $X$. 

In fact, we will only need the above assertions about $N_{f,0}$ in the (easier) case when $f$ defines a simple normal crossing divisor. 
Under this assumption, we deduce the same assertion for all $N_{f,m}$, as follows.

\begin{prop}\label{N_f}
Suppose that $f\in\cO_X(X)$ satisfies (\ref{eq_assumption}) and defines a simple normal crossing divisor and let $m\geq 0$.
\begin{enumerate}
\item[i)] The $\cD_X$-module $N_{f,m}$ is generated by $\partial_t^j f^s$, for $0\leq j\leq m$.
\item[ii)] The characteristic variety of $N_{f,m}$ is $W_f$.
\end{enumerate}
\end{prop}

\begin{proof}
In order to prove i), note first that by (\ref{action_of_t}), for every $j\geq 0$ we have
$$t\partial_t^jf^s=f\partial_t^jf^s-j\partial_t^{j-1}f^s.$$
We deduce by descending induction on $0\leq j\leq m$ that $\partial_t^jf^s\in N_{f,m}$. 
Moreover, the same formula implies that if we put
$N'_{f,m}=\sum_{j=0}^m\cD_X\cdot \partial_t^jf^s$, then $N'_{f,m}$ is a $\cD_X[t]$-submodule of $N_{f,m}$. 
In order to complete the proof of i) it is enough to show that $s\partial_t^jf^s\in N'_{f,m}$ for $0\leq j\leq m$.
In fact, it follows from (\ref{eq1_lem2}) that
$$s\partial_t^jf^s=\partial_t^j(s+j)f^s,$$
hence it is enough to show that $sf^s\in\cD_Xf^s$.

This is a local assertion. Since $f$ defines a simple normal crossing divisor, after passing to the elements of a suitable affine open cover of $X$, 
we may assume that $X$ is affine and we have algebraic coordinates $x_1,\ldots,x_n$ on $X$ such that $f=px_1^{a_1}\cdots x_n^{a_n}$, for some nonnegative integers 
$a_1,\ldots,a_n$, with $p$ an invertible regular function. Since 
$$x_i\partial_{x_i}f^s=s\left(x_ip^{-1}\frac{\partial p}{\partial x_i}+a_i\right)f^s,$$
it follows that in order to complete the proof of i) it is enough to show that the functions
$x_ip^{-1}\frac{\partial p}{\partial x_i}+a_i$, with $1\leq i\leq n$, have no common zeroes in $X$. 
Now, it is straightforward to check that condition (\ref{eq_assumption}) implies that the zero locus of these functions
is contained in the zero-locus of $f$. On the other hand, it can't intersect the zero-locus of $f$: if $x_i$ vanishes at a point in $X$ and $a_i>0$,
then $x_ip^{-1}\frac{\partial p}{\partial x_i}+a_i$ can't vanish at that point. This implies that indeed, the zero-locus of 
the functions $x_ip^{-1}\frac{\partial p}{\partial x_i}+a_i$, with $1\leq i\leq n$, is empty.

For ii), recall first that we know by  \cite[Theorem~5.3]{Kashiwara} that the characteristic variety ${\rm Char}(N_{f,0})$ of $N_{f,0}$ is equal to $W_f$.
We prove the general case by induction on $m$. If $m\geq 1$, then it follows from i) that we have
$N_{f,m-1}\subseteq N_{f,m}$ and the quotient $N_{f,m}/N_{f,m-1}$ is generated over $\cD_X$ by the class $w_m$ of $\partial_t^mf^s$. 
This implies that we have a surjective map $N_{f,0}=\cD_Xf^s\to N_{f,m}/N_{f,m-1}$ that maps $Pf^s$ to $Pw_m$ for $P\in\cD_X$
(note that if $Pf^s=0$, then
$P\partial_t^mf^s=\partial_t^mPf^s=0$ in $N_{f,m}$). 

We thus conclude, using also the induction hypothesis, that we have 
$$W_f={\rm Char}(N_{f,m-1})\subseteq {\rm Char}(N_{f,m-1})\cup {\rm Char}\big(N_{f,m}/N_{f,m-1})={\rm Char}(N_{f,m})$$
$$\subseteq
{\rm Char}(N_{f,m-1})\cup {\rm Char}(N_{f,0})=W_f.$$
Therefore ${\rm Char}(N_{f,m})=W_f$.
\end{proof}

As in \cite{Lichtin}, we will be making use also of right $\cD$-modules. Recall that if $\cM$ is a left $\cD_X$-module, then on the
$\cO_X$-module $\omega_X\otimes_{\cO_X}\cM$ one can put a right $\cD_X$-module structure. In this way, one gets a equivalence
of categories between left and right $\cD_X$-modules. 

This is easy to describe when we have a system of coordinates $x_1,\ldots,x_n$ on $X$. Indeed, in this case we have an involution
$\cD_X\to \cD_X$, denoted $P\to P^*$, uniquely characterized by the fact that $(PQ)^*=Q^*P^*$, $f^*=f$ for $f\in\cO_X$, and
$\partial_{x_i}^*=-\partial_{x_i}$. If $u$ is a section of $\cM$ and we denote by $u^*$ the section
$dx_1\wedge\ldots\wedge dx_n\otimes u$ of $\omega_X\otimes_{\cO_X}\cM$, then we have $u^*P^*=(Pu)^*$ for every section $P$ of $\cD_X$. 

Similarly, we have an equivalence of categories between left and right $\cD_X\langle s,t\rangle$-modules, which takes
a left $\cD_X\langle s,t\rangle$-module $\cM$, to $\omega_X\otimes_{\cO_X}\cM$. Again, the right action on this $\cO_X$-module
is easy to describe when we have coordinates $x_1,\ldots,x_n$: the involution described above on $\cD_X$ extends to a similar involution on
$\cD_X\langle t,\partial_t\rangle$ which maps $t$ to $t$ and $\partial_t\to -\partial_t$, hence maps $s=-\partial_tt$ to
$t\partial_t=\partial_tt-1$. This restricts to an involution on $\cD_X\langle s,t\rangle$, still denoted $P\to P^*$, that maps $t$ to $t$ and $s$ to $-s-1$,
such that $u^*P^*=(Pu)^*$ for every sections $u$ of $\cM$ and $P$ of $\cD_X\langle s,t\rangle$. 

\begin{eg}\label{eg1}
Suppose that we have coordinates $x_1,\ldots,x_n$ on $X$ and we consider the $\cD_X\langle s,t\rangle$-module $\widetilde{B}_f$. 
If $u$ is a section of $\widetilde{B}_f$ and $u^*$ is the corresponding section of $\omega_X\otimes_{\cO_X}\widetilde{B}_f$, then we see
that $b_u(s)$ is equal to $b_{u^*}(-s-1)$, where $b_{u^*}(s)$ is 
the monic polynomial of minimal degree with the property that
$$u^*b_{u^*}(s)\in u^*t\cdot\cD_X\langle s,t\rangle.$$
\end{eg}

We can now give the proof of our main result.

\begin{proof}[Proof of Theorem~\ref{rootbound}]
Note that the assertions in the theorem are local on $X$ (see Remark~\ref{local_char}). After taking a suitable affine open cover of $X$, we may
assume that $X$ is affine and that we have algebraic coordinates $x_1,\ldots,x_n$ on $X$. We put $dx=dx_1\wedge\ldots\wedge dx_n$.

If $f$ is invertible, then $b_u=1$ for every $u\in \widetilde{B}_f$. Indeed, 
suppose that $m$ is such that $u\in \oplus_{i=0}^m\cO_X[1/f]\partial_t^if^s$.
It follows from
(\ref{action_of_t}) that $(t-f)^{m+1}u=0$. Since $f$ is invertible, dividing by $f^{m+1}$,
we see that there is $P\in\cO_X[t]$ such that $u=tP\cdot u$, hence $b_u=1$. 
Therefore in this case all assertions in the theorem are trivial.

From now
on, we assume that $f$ is not invertible; in particular, it is not a constant. Moreover, this shows that we may always replace $X$ by an open neighborhood
of the zero-locus of $f$. Since $f\circ \pi$ is not constant, by applying the Generic Smoothness theorem for the map $f\circ \pi\colon Y\to \C$,
we see that after possibly 
removing from $X$ finitely many closed subsets of the form $f^{-1}(\lambda)$, for $\lambda\neq 0$, we may assume that $f\circ \pi$ satisfies condition (\ref{eq_assumption}). 

After these preparations, we begin the proof of the theorem, following closely the argument in \cite{Kashiwara}, as modified in \cite{Lichtin}. 
We need to understand the roots of $b_u(s)$, where $u=g\partial_t^mf^s$. As in Example~\ref{eg1}, we consider the right
$\cD_X\langle s,t\rangle$-module $N:=\omega_X\otimes_{\cO_X}N_{f,m}(g)$ and the global section $u^*=dx\otimes u$. We have seen that
$b_u(s)=b_{u^*}(-s-1)$, where $b_{u^*}(s)$ is the monic polynomial of minimal degree such that
$$u^*b_{u^*}(s)\in u^*t\cdot\cD_X\langle s,t\rangle.$$
Arguing as in Remark~\ref{rem_interpretation_b_function}, we see that $b_{u^*}$ is the minimal polynomial of the action of $s$ on 
the right $\cD_X$-module $N/Nt$. 

Let $F=f\circ\pi$ and $G=g\circ \pi$. On $Y$ we have the right $\cD_Y\langle s,t\rangle$-module $\omega_Y\otimes_{\cO_Y}N_{F,m}(G)$
and its global section $v=\pi^*(dx)\otimes G\partial_t^mF^s$. We consider $M:=v\cdot \cD_Y\langle s,t\rangle$ and denote by $B(s)$ the minimal
polynomial of the action of $s$ on the $\cD_Y$-module $M/Mt$. Suppose that $U$ is an affine open subset of $Y$ on which we have coordinates
$y_1,\ldots,y_n$ such that
$$F\vert_U=p_1y_1^{a_1}\cdots y_n^{a_n}\quad\text{and}\quad
\pi^*(dx)=p_2y_1^{k_1}\cdots y_n^{k_n}dy,$$
with $p_1$ and $p_2$ invertible functions on $U$. 
We can also write $G\vert_U=hy_1^{b_1}\cdots\cdot y_m^{b_m}$, for some $h\in\cO_Y(U)$, with the $b_j$ as in the statement of the theorem.

If we write $B_U(s)$ for the minimal polynomial of the action of $s$ on the $\cD_U$-module $M/Mt\vert_U$, 
after translating to left $\cD$-modules and using Lemma~\ref{lem1}, we see that $B_U(s)$ 
is equal to 
$q_U(-s-1)$, where $q_U(s)$ is the $b$-function corresponding to the element 
$$h\cdot \prod_{j=1}^ny_j^{k_j+b_j}\partial_t^m(y_1^{a_1}\cdots y_n^{a_n})^s\in\widetilde{B}_{y_1^{a_1}\cdots y_n^{a_n}}.$$
We now apply Lemma~\ref{lem3} to conclude the following:
\begin{enumerate}
\item[i)] Every root of $q_U$ is $\leq -\min\left\{1,-m+\min_{a_i\neq 0}\frac{k_i+1+b_i}{a_i}\right\}$.
\item[ii)] If $m=0$, then every root of $q_U$ is $\leq -\min_{a_i\neq 0}\frac{k_i+1+b_i}{a_i}$.
\item[iii)] If $g=1$, then every root of $q_U$ is either equal to $-1$ or to some $m-\frac{k_i+\ell}{a_i}$, with $a_i\neq 0$ and
$1\leq\ell\leq a_i$. Moreover, if $D$ is reduced and its strict transform on $Y$ is smooth, then we may assume that the divisor on
$Y$ defined by $y_i$ is exceptional (indeed, in this case at most one $y_i$ satisfies $k_i=0$--equivalently, it is not exceptional--and $a_i>0$; moreover, in this case $a_i=1$). 
\end{enumerate}
Note that the properties in i) and ii) are clear if $h=1$. By Theorem~\ref{thm_Sabbah}, they are equivalent to
$$\prod_{j=1}^ny_j^{k_j+b_j}\partial_t^m(y_1^{a_1}\cdots y_n^{a_n})^s\in V^{\gamma}\widetilde{B}_{y_1^{a_1}\cdots y_n^{a_n}},$$ 
where $\gamma$ is the respective minimum. However, the pieces of the $V$-filtration are $\cO_U$-modules, which implies that
$$h\cdot \prod_{j=1}^ny_j^{k_j+b_j}\partial_t^m(y_1^{a_1}\cdots y_n^{a_n})^s\in V^{\gamma}\widetilde{B}_{y_1^{a_1}\cdots y_n^{a_n}},$$ 
giving the assertions in i) and ii).

Finally, we note that if we consider a cover of $Y$ by affine open subsets $U$ as above, then $B(s)$ is the least common multiple of the polynomials
$B_U(s)$. The rest of the proof is devoted to showing that every root of $b_{u^*}(s)$ is of the form $\lambda+j$ for some root $\lambda$ of $B$
and some non-negative integer $j$. In light of properties i), ii), and iii) above, this implies the statement of the theorem. 

In order to relate $B$ and $b_{u^*}$, note first that the action of $t$ on $\widetilde{B}_F$ is injective and thus the (right) action of $t$ on 
$M$ is injective. Since $M\cdot B(s)\subseteq Mt$, it follows that there is a morphism of right $\cD_Y$-modules $\varphi\colon M\to M$ such that 
\begin{equation}\label{eq1_proof_main}
B(s)=t\circ\varphi.
\end{equation}
 We now consider the coherent right $\cD_X$-module $\int^0_{\pi}M$, the $0^{\rm th}$ cohomology of the $\cD$-module theoretic 
push-forward of $M$ by $\pi$. Since $M$ is a right $\cD_Y\langle s,t\rangle$-module, we see that $\int^0_{\pi}M$ is a right $\cD_X\langle s,t\rangle$-module.
By functoriality, the equality (\ref{eq1_proof_main}) gives the equality $B(s)=t\circ\int^0_{\pi}\varphi$ of maps on $\int^0_{\pi}M$. We thus see that
\begin{equation}\label{eq2_proof_main}
\left(\int^0_{\pi}M\right)\cdot B(s)\subseteq\left(\int^0_{\pi}M\right)\cdot t.
\end{equation}

A key ingredient for what follows is a canonical section of $\int^0_{\pi}M$, that can be defined as follows. Recall that by definition, we have
$$\int^0_{\pi}M=R^0\pi_*(M\otimes^L_{\cD_Y}\cD_{Y\to X}),$$
where $\cD_{Y\to X}$ is the transfer bimodule $\pi^*(\cD_X)$. Note now that on $Y$ we have a morphism of right $\cD_Y$-modules 
$\cD_Y\to M$ that maps $1$ to $v$. The (derived) tensor product with $\cD_{Y\to X}$ induces the morphism
$$\cD_{Y\to X}\to M\otimes^L_{\cD_Y}\cD_{Y\to X}.$$
On the other hand, the global section $1$ of $\cD_X$ induces a global section $\pi^*(1)$ of $\cD_{Y\to X}$, hence a morphism
of $\cO_Y$-modules $\cO_Y\to \cD_{Y\to X}$. Applying $R^0\pi_*$ to the composition
$$\cO_Y\to\cD_{Y\to X}\to M\otimes^L_{\cD_Y}\cD_{Y\to X}$$
gives a morphism 
$$\cO_X=R^0\pi_*(\cO_Y)\to R^0\pi_*(M\otimes^L_{\cD_Y}\cD_{Y\to X})=\int_{\pi}^0M.$$
The image of $1$ is a global section of $\int_{\pi}^0M$, that we denote $w$. It is straightforward to see that if $V\subseteq X$ is the complement
of the zero-locus of $f$ (so that, by assumption, $\pi$ is an isomorphism over $V$),
then $\int^0_{\pi}M\vert_V=\big(\omega_X\otimes_{\cO_X}N_{f,m}(g)\big)\vert_V$ and $w\vert_V=u^*\vert_V$. 

Let $L=w\cdot\cD_X\langle s,t\rangle\subseteq\int^0_{\pi}M$. We next show that there is a morphism of right $\cD_X\langle s,t\rangle$-modules
$\psi\colon L
\to\omega_X\otimes_{\cO_X}N_{f,m}(g)$ that maps $w$ to $u^*$ (in which case $\psi$ is clearly surjective). For this, it is enough to show that if
$P$ is a section of $\cD_X\langle s,t\rangle$ such that $w\cdot P(s,t)=0$, then $u^*\cdot P(s,t)=0$. This follows from the fact that on $V$ the section $w$
gets identified with $u^*$ and the fact that $\omega_X\otimes_{\cO_X}N_{f,m}(g)$ is a torsion-free $\cO_X$-module (this follows from the fact that
$\omega_X\otimes_{\cO_X}\widetilde{B}_f$ is isomorphic as an $\cO_X$-module to $\cO_X[1/f,s]$, hence it is torsion-free). 

We next show that the quotient $M':=(\int^0_{\pi}M)/L$ is holonomic as a $\cD_X$-module.
Since $N_{F,m}(G)\subseteq N_{F,m}$ and $F$ satisfies (\ref{eq_assumption})
we have by Proposition~\ref{N_f}
$${\rm Char}\big(N_{F,m}(G)\big)\subseteq {\rm Char}(N_{F,m})=W_F.$$
Of course, the characteristic variety of a right $\cD$-module is the characteristic variety of the corresponding left $\cD$-module,
hence ${\rm Char}(M)\subseteq {\rm Char}\big(N_{F,m}(G)\big)\subseteq W_F$.
It then follows from \cite[Theorem~4.2]{Kashiwara} that if $\alpha\colon Y\times_XT^*X\to T^*X$ and $\beta\colon Y\times_XT^*X\to T^*Y$
are the canonical morphisms, then 
$${\rm Char}\left(\int^0_{\pi}M\right)\subseteq \alpha\big(\beta^{-1}\big({\rm Char}(M)\big)\big)\subseteq \alpha\big(\beta^{-1}(W_F)\big).$$

We have seen that the restriction of $M'$ to $V$ is $0$, hence 
$${\rm Char}(M')\subseteq {\rm Char}\left(\int^0_{\pi}M\right)\cap \big(T^*X\times_X(X\smallsetminus V)\big).$$
We deduce that indeed $M'$ is holonomic
if we show that $\alpha\big(\beta^{-1}(W_F)\big)$ is contained in the union of $W_f$ with some subvarieties of $T^*X$
of dimension $\leq n$ (recall that $W_f$ is irreducible, of dimension $n+1$, and dominates $X$).
Let us write
\begin{equation}\label{eq_union}
\alpha\big(\beta^{-1}(W_F)\big)=\big(\alpha(\beta^{-1}(W_F))\times_XV\big)\cup \big(\alpha(\beta^{-1}(W_F))\times_X(X\smallsetminus V)\big).
\end{equation}
Since $\pi$ is an isomorphism over $V$, the first term in the union is equal to $W_f\times_XV\subseteq W_f$. On the other hand,
the second term in the union has dimension $\leq n$: it is shown in \cite[Proposition~5.6]{Kashiwara} that 
$W_F\times_Y\big(Y\smallsetminus \pi^{-1}(V)\big)$ is isotropic with respect to the canonical symplectic structure on $T^*Y$, hence by 
\cite[Proposition~4.9]{Kashiwara}, also $\alpha\big(\beta^{-1}(W_F)\big)\times_X(X\smallsetminus V)\subseteq T^*X$ is isotropic, hence of dimension $\leq n$. 
We thus conclude that $M'$ is a holonomic $\cD_X$-module.

Since $M'$ is a $\cD_X\langle s,t\rangle$-module which is holonomic
as a $\cD_X$-module, it follows from \cite[Proposition~5.11]{Kashiwara} that there is $N\geq 0$ such that 
$$\left(\int^0_{\pi}M\right)t^{N}\subseteq L, \quad\text{hence}\quad \left(\int^0_{\pi}M\right)t^{N+1}\subseteq L\cdot t.$$
On the other hand, it follows from (\ref{eq2_proof_main}), using the relations (\ref{eq1_lem2}) that 
$$L\cdot B(s)B(s-1)\cdots B(s-N)\subseteq \left(\int^0_{\pi}M\right)\cdot B(s)B(s-1)\cdots B(s-N)\subseteq \left(\int^0_{\pi}M\right)t^{N+1}\subseteq L\cdot t.$$
Finally, since $\omega_X\otimes_{\cO_X}N_{f,m}(g)$ is a quotient of $L$, it follows that
$$\omega_X\otimes_{\cO_X}N_{f,m}(g)\cdot B(s)B(s-1)\cdots B(s-N)\subseteq \big(\omega_X\otimes_{\cO_X}N_{f,m}(g)\big)\cdot t,$$
hence $b_{u^*}(s)$ divides $\prod_{j=0}^NB(s-j)$.
We thus see that every root of $b_{u^*}(s)$ is of the form $\lambda+j$, for some root $\lambda$ of $B$
and some non-negative integer $j$. As we have seen, this gives the assertions in the theorem.
\end{proof}

\begin{rmk}
The reason why in Theorem~\ref{rootbound} we don't get the same precise statement for the roots of the $b$-function of $b_{g\partial_t^mf^s}$ 
when $g\neq 1$ is that in the proof of the theorem we need $\pi$ to be an isomorphism over the complement $V$ of the zero-locus of $f$. However,
if $g$ satisfies the condition that its restriction to $V$ is a simple normal crossing divisor, then we can find $\pi$ such that, in addition, the inverse image 
of the hypersurface defined by $fg$ is a simple normal crossing divisor. In this case, the proof of Theorem~\ref{rootbound} (together with the corresponding
assertions in Lemma~\ref{lem3}) imply that 
\begin{enumerate}
\item[i)] Every root of $b_{g\partial_t^mf^s}$ is either a negative integer, or of the form $m-\frac{k_i+1+b_i+\ell}{a_i}$, for some 
nonnegative integer $\ell$.
\item[ii)] Moreover, if $m=0$, then the root is necessarily of the form $-\frac{k_i+1+b_i+\ell}{a_i}$, for some 
nonnegative integer $\ell$.
\end{enumerate}
\end{rmk}

\section{Multiplier ideals and the $V$-filtration}

We begin with a brief review of multiplier ideals. For a more in-depth discussion, we refer to \cite[Chapter 9]{Lazarsfeld}. Suppose that $X$ is a smooth
complex algebraic variety and
the nonzero $f\in\cO_X(X)$ defines the effective divisor
$D$. If $\pi\colon Y\to X$ is a log resolution of the pair $(X,D)$ (note that in this setting we don't need
$\pi$ to be an isomorphism over the complement of the support of $D$), then for every $\lambda\in\Q_{>0}$, we have
$$\cI(f^{\lambda})=\pi_*\cO_Y\big(K_{Y/X}-\lfloor\lambda \pi^*(D)\rfloor\big).$$
Here, if $\pi^*(D)=\sum_{i=1}^ra_iE_i$, we put $\lfloor\lambda D\rfloor=\sum_{i=1}^r\lfloor\lambda a_i\rfloor E_i$,
where for a real number $\alpha$, we denote by $\lfloor\alpha\rfloor$ the largest integer that is $\leq\alpha$.
If we write $K_{Y/X}=\sum_{i=1}^rk_iE_i$, then for $g\in\cO_X(X)$ we have that $g$ is a section of $\cI(f^{\lambda})$
if and only if $b_i+k_i-\lfloor\lambda a_i\rfloor\geq 0$ for all $i$, where $b_i={\rm ord}_{E_i}(g)$. 
In other words, this is the case if and only if
\begin{equation}\label{formula_lct}
\lambda<\min_i\frac{k_i+1+b_i}{a_i}
\end{equation}
(of course, we make the convention that the quotient is $\infty$ if $a_i=0$).
 We denote the right-hand side of (\ref{formula_lct}) by 
${\rm lct}_g(f)$. Note that for $g=1$, we recover the log canonical threshold ${\rm lct}(f)$ of $f$. 

We will also need the analytic description of multiplier ideals, which we now recall. Suppose that $z_1,\ldots,z_n$ are algebraic (or analytic)
coordinates on an open subset $U$ of $X$. Using these coordinates, we obtain an isomorphism of $U$ with an open subset of $\C^n$. In particular,
we get an induced Lebesgue measure on $U$. Given $g\in\cO_X(U)$, we have $g\in\Gamma\big(U,\cI(f^{\lambda})\big)$ if and only if 
the function $\frac{|g|^2}{|f|^{2\lambda}}$ is locally integrable on $U$. 

After these preparations, we can give the first application of Theorem~\ref{rootbound}.

\begin{proof}[Proof of Corollary~\ref{equality}]
Let $\alpha={\rm lct}_g(f)$ and $b(s)=b_{gf^s}(s)$. We deduce from Theorem~\ref{rootbound}ii) that every root of $b(s)$ is $\leq -\alpha$.
The assertion in the corollary thus follows if we prove that  $b(-\alpha)=0$. In order to show this, we follow closely the proof
of \cite[Theorem~B]{ELSV}, which in turn is based on the proof of \cite[Theorem~10.6]{Kollar}. 

Arguing by contradiction, let us assume that $b(-\alpha)\neq 0$. We will show that in this case $g$ is a section of $\cI(f^{\alpha})$, a contradiction.
After covering $X$ by suitable affine open subsets, we may assume that $X$ is an open subset $U\subseteq\C^n$.
We write $z_1,\ldots,z_n$ for the coordinate functions on $\C^n$.
 By the definition of ${\rm lct}_g(f)$,
and using the analytic interpretation of the multiplier ideals, we know that the function $\frac{|g|^2}{|f|^{2\lambda}}$ is locally integrable for $\lambda<\alpha$;
we aim to show that it is also locally integrable for $\lambda=\alpha$, so that we get the desired contradiction.

By definition of $b(s)$, there is a differential operator $P\in \Gamma(U,\cD_U)[s]$ such that
$$b(s)gf^s=P(s)\cdot (gf)f^s.$$
By specializing $s$ to $-\lambda$, and letting $Q=\frac{1}{b(-\lambda)}P(-\lambda)\in \Gamma(U,\cD_U)$ (recall that by assumption $b(-\lambda)\neq 0$), we have
\begin{equation}\label{eq1_pf_equality}
gf^{-\lambda}=Q\cdot gf^{1-\lambda}.
\end{equation}
If we apply complex conjugation, we get
\begin{equation}\label{eq2_pf_equality}
\overline{g}\overline{f}^{-\lambda}=\overline{Q}\cdot \overline{g}\overline{f}^{1-\lambda}.
\end{equation}
Note now that since $\partial_{z_i}(\overline{h})=0$ for every holomorphic function on $U$, given two such holomorphic
functions $h_1$ and $h_2$, we have $Q\cdot (\overline{h_1}h_2)=\overline{h_1}\cdot (Q\cdot h_2)$. Similarly, we have
$\overline{Q}\cdot (\overline{h_1}h_2)=h_2\cdot (\overline{Q}\cdot \overline{h_1})$. We thus see that if
$R=Q\overline{Q}$, then by multiplying (\ref{eq1_pf_equality}) and (\ref{eq2_pf_equality}), we obtain
\begin{equation}\label{eq3_pf_equality}
|g|^2|f|^{-2\lambda}=R\cdot \big(|g|^2|f|^{2(1-\lambda)}\big).
\end{equation}

Note now that the function $\frac{|g|^2}{|f|^{2\lambda}}$ is locally integrable if and only if for every ${\mathcal C}^{\infty}$ function $\varphi$ on $U$, with compact support, the function 
$$U\ni z\to \frac{|g(z)|^2}{|f(z)|^{2\lambda}}\varphi(z)$$
is integrable. On the other hand, if we denote by $\widetilde{R}$ the adjoint\footnote{Recall that if we choose real coordinates 
$x_1,\ldots,x_{2n}$ on $\C^n$ and if we write $R=\sum_{\alpha}h_{\alpha}\partial_x^{\alpha}$, then the adjoint of $R$ is
$\sum_{\alpha}(-\partial_x^{\alpha})h_{\alpha}$.} of $R$, we deduce from the Stokes theorem, using the fact that $\varphi$ has compact support in $U$, that
$$\int_U\left(R\cdot \frac{|g(z)|^2}{|f(z)|^{2(\lambda-1)}}\right)\varphi(z)dzd\overline{z}
=\int_U\frac{|g(z)|^2}{|f(z)|^{2(\lambda-1)}}(\widetilde{R}\cdot\varphi)dzd\overline{z},$$
in the sense that one integral is well-defined if and only if the other one is, and in this case they are equal. 
However, since $\widetilde{R}\cdot\varphi$ is a ${\mathcal C}^{\infty}$ function on $U$, with compact support, and since by assumption
we know that the function $\frac{|g(z)|^2}{|f(z)|^{2(\lambda-1)}}$ is locally integrable, it follows that the right-hand side in the above formula is
well-defined. We thus conclude using (\ref{eq3_pf_equality}) that the function $ \frac{|g(z)|^2}{|f(z)|^{2\lambda}}\varphi(z)$ is integrable. Since this holds for every $\varphi$,
we have obtained a contradiction.
\end{proof}

We now translate this to the theorem of Budur and Saito \cite{BS} describing multiplier ideals via the $V$-filtration.

\begin{proof}[Proof of Corollary~\ref{BSthm}]
It follows from Theorem~\ref{thm_Sabbah} that $gf^s$ lies in $V^{\alpha}\widetilde{B}_f$ if and only if all roots of $b_{gf^s}(s)$ are $\leq-\alpha$.
By Corollary~\ref{equality}, this is equivalent to ${\rm lct}_g(f)\geq\alpha$, which in turn is equivalent to $g$ being a section of $\cI(f^{\lambda})$
for all $\lambda<\alpha$. 
\end{proof}

\section{A lower bound for the minimal exponent}

In this section we deduce Corollary~\ref{bound_min_exp}
from our main result. We assume that $f\in\cO_X(X)$ defines a reduced, nonzero divisor $D$ on the smooth variety $X$.
In this case, the Bernstein-Sato polynomial of $f$ can be written as $b_f(s)=(s+1)\widetilde{b}_f(s)$ (see Remark~\ref{reduced_b_function}).
Recall from the Introduction that the negative of the largest root of $\widetilde{b}_f(s)$ is the \emph{minimal exponent} $\widetilde{\alpha}(f)$ (with the convention that this is infinite
if $\widetilde{b}_f(s)=1$).

We consider a strong log resolution $\pi\colon Y\to X$ for $D$ with the property that the strict transform of $D$ on $Y$ is smooth and
use the notation in Theorem~\ref{rootbound}.

\begin{proof}[Proof of Corollary~\ref{bound_min_exp}]
Let us write 
$$\min_{i; E_i \text{exceptional}}\frac{k_i+1}{a_i}=m+\alpha,$$
for a nonnegative integer $m$ and $\alpha\in (0,1]$. 
The key point is that by a result of Saito (see \cite[(1.3.8)]{Saito-MLCT}, where this is phrased in terms of the \emph{microlocal $V$-filtration})
 $\widetilde{\alpha}(f)\geq m+\alpha$
if and only if $\partial_t^mf^s\in V^{\alpha}\widetilde{B}_f$. By Theorem~\ref{thm_Sabbah}, this is equivalent with the fact that
all roots of $b_{\partial_t^mf^s}(s)$ are $\leq -\alpha$. 

Alternatively, this can be seen as follows: it was shown in \cite[Proposition~6.12]{MP} that 
$$b_{\partial_t^mf^s}(s)\vert (s+1)\widetilde{b}_f(s-m)\quad\text{and}\quad \widetilde{b}_f(s-m)\vert b_{\partial_t^mf^s}(s).$$
Since $\alpha\leq 1$, this implies that every root of $\widetilde{b}_f(s)$ is $\leq -(m+\alpha)$ if and only if every root of
$b_{\partial_t^mf^s}(s)$ is $\leq -\alpha$.

By Theorem~\ref{rootbound}, we know that every root $\lambda$ of $b_{\partial_t^mf^s}(s)$ is either a negative integer (which is $\leq-\alpha$, since
$\alpha\leq 1$) or equal to $m-\frac{k_i+1+\ell}{a_i}$, for some exceptional divisor divisor $E_i$ and some nonnegative integer $\ell$.
Since $\frac{k_i+1+\ell}{a_i}\geq m+\alpha$, we conclude that $\lambda\leq -\alpha$. This completes the proof of the corollary.
\end{proof}

\end{document}